\documentclass[oneside,english]{amsart}
\usepackage[T1]{fontenc}
\usepackage[latin9]{inputenc}
\usepackage{amsthm}
\usepackage{amstext}
\usepackage{amssymb}
\usepackage{esint}

\makeatletter

\DeclareFontEncoding{LGR}{}{}
\DeclareTextSymbol{\~}{LGR}{126}
\newcommand{\lyxmathsym}[1]{\ifmmode\begingroup\def\b@ld{bold}
  \text{\ifx\math@version\b@ld\bfseries\fi#1}\endgroup\else#1\fi}

\numberwithin{equation}{section}
\numberwithin{figure}{section}
\theoremstyle{plain}
\newtheorem{thm}{\protect\theoremname}[section]
  \theoremstyle{plain}
  \newtheorem{lem}[thm]{\protect\lemmaname}
  \theoremstyle{definition}
  \newtheorem{defn}[thm]{\protect\definitionname}
  \theoremstyle{plain}
  \newtheorem{prop}[thm]{\protect\propositionname}
  \theoremstyle{plain}
  
  \theoremstyle{definition}
  
	\newtheorem{rem}[thm]{\protect\remarkname}
  \theoremstyle{plain}

\makeatother

\usepackage{babel}
  \providecommand{\corollaryname}{Corollary}
  \providecommand{\definitionname}{Definition}
  \providecommand{\examplename}{Example}
  \providecommand{\lemmaname}{Lemma}
  \providecommand{\propositionname}{Proposition}
\providecommand{\theoremname}{Theorem}
\providecommand{\remarkname}{Remark}

\begin{document}

\centerline{\bf Conformal Spectral Stability Estimates for the Neumann Laplacian}

\vskip 0.3cm
\centerline{\bf V.~I.~ Burenkov, V.~Gol'dshtein, A.~Ukhlov}

\vskip 0.3cm

\centerline{\small ABSTRACT}
\vskip 0.1cm
{\small 
We study the eigenvalue problem for the Neumann-Laplace operator in conformal regular planar domains $\Omega\subset\mathbb{C}$.
Conformal regular domains support the Poincar\'e inequality and this allows us to estimate the variation of the eigenvalues of the 
Neumann Laplacian upon domain perturbation via energy type integrals. Boundaries of such domains can have any Hausdorff dimension between one and two.
}
\vskip 0.1cm



\footnotetext{{\bf Key words and phrases:} eigenvalue problem, elliptic equations, conformal mappings, quasidiscs.}
\footnotetext{{\bf 2010 Mathematics Subject Classification:} 35P15, 35J40, 47A75, 47B25.}

\section{Introduction }

This paper is devoted to stability estimates for the  eigenvalues of the Neumann Laplacian
$$
-\Delta f = -\Big(\frac{\partial^2 f}{\partial x^2}+\frac{\partial^2 f}{\partial y^2}\Big),\,\,\, (x,y)\in \Omega,\,\,\,\,\, {\frac{\partial f}{\partial n}}\bigg|_{\partial\Omega}=0,
$$
in conformal regular planar domains $\Omega\subset\mathbb C$.

In \cite{GU15} was proved (see, also, \cite{KOT}) that in such domains the embedding operator
$$
i_{\Omega}: W^{1,2}(\Omega)\to L^2(\Omega)
$$
is compact. Hence in the conformal regular planar domains $\Omega\subset\mathbb{C}$ the spectrum of the Neumann Laplacian is discrete and can be written in the form
of a non-decreasing sequence
\[
0=\lambda_{1}[\Omega]<\lambda_{2}[\Omega]\leq...\leq\lambda_{n}[\Omega]\leq...\,,
\]
where each eigenvalue is repeated as many times as its multiplicity.

The spectral stability estimates for the Laplace operator were intensively studied in the last two decades. See, for example, the survey papers \cite{BLL}, \cite{B} where one can found the history of the problem, main results in this area and appropriate references.

Recall that $\Omega$ is called a conformal $\alpha$-regular domain, $\alpha>2$, \cite{BGU15} if there exists a conformal mapping 
$\varphi$ of the unit disc $\mathbb{D}$ onto $\Omega$ such that $\varphi'\in L^{\alpha}(\mathbb D)$.

A domain $\Omega$ is a conformal regular domain if it is a conformal $\alpha$-regular domain for some $\alpha>2$.

The notion of conformal $\alpha$-regular domains (conformal regular domains) does not depends on a choice of a conformal homeomorphism $\varphi$  \cite{BGU15}.

Note that any conformal regular domain has finite geodesic diameter 
\cite{GU5} and can be characterized in the terms of the (quasi)\-hy\-per\-bolic boundary condition \cite{BP, KOT}. 

The notion of conformal regularity was introduced in \cite{BGU15} for study of the spectral stability of the Dirichlet-Laplace operator. 
Let, for $2<\alpha\le\infty, \tau>0$, $G_{\alpha,\tau}$ be the set of all conformal mappings $\varphi$ of the unit disc $\mathbb D$ such that
$$
\|\varphi'\mid {L^{\alpha}(\mathbb D)}\|\le\tau\,.
$$

\begin{thm} \cite{BGU15} For any $2<\alpha\le\infty, \tau>0$ there exists $B_{\alpha,\tau}>0$ such that for any $\varphi_1,\varphi_2\in G_{\alpha,\tau}$ and for any $n\in\mathbb N$ the eigenvalues $\tilde{\lambda}_{n}$ of the Dirichlet-Laplace operator in domains $\Omega_1$ and $\Omega_2$ satisfy the inequality
$$
|\tilde{\lambda}_{n}[\Omega_{1}]-\tilde{\lambda}_{n}[\Omega_{2}]|\\
\leq c_n B_{\alpha,\tau} \|\varphi_1-\varphi_2\mid L^{1,2}(\mathbb D)\|\,,
$$
where $\Omega_{1}=\varphi_1(\mathbb D)$, $\Omega_{2}=\varphi_2(\mathbb D)$ and $c_{n}=\max\{\tilde{\lambda}_{n}^{2}[\Omega_{1}],\tilde{\lambda}_{n}^{2}[\Omega_{2}]\}$.
\end{thm}
Let us remark that the constant $B_{\alpha,\tau}$ depends on conformal geometry of domains $\Omega_1$ and $\Omega_2$ only \cite{BGU15}:
$$
B_{\alpha,\tau}=\left[\tilde{C}\left(\alpha\right)\right]^2\Big(\|\varphi_1'\mid L^{\alpha}(\mathbb D)\|+
\|\varphi_2'\mid L^{\alpha}(\mathbb D)\|\Big),
$$
where the constant $\tilde{C}(\alpha)$ is the best constant for corresponding Poincar\'e-Sobolev inequality in the unit disc.

In the case of the Neumann-Laplace operator we need an additional assumption on domains. 

\begin{defn}
Two conformal $\alpha$-regular domains   $\Omega_1,\Omega_2$  represent conformal $\alpha$-regular pair if there exists a conformal mapping $\psi:\Omega_1\to\Omega_2$ such that 
\begin{equation}
\label{alpharegeq}
\iint\limits_{\Omega_1}|(\psi'(z)|^{\alpha}~dxdy<\infty\,\,\&\,\,\iint\limits_{\Omega_2}|(\psi^{-1})'(w)|^{\alpha}~dudv<\infty
\end{equation}
for some $\alpha>2$.
\end{defn}

Two conformal regular domains  $\Omega_1,\Omega_2$ represent conformal regular pair if they are  conformal $\alpha$-regular pair for some $\alpha>2$.

\begin{prop} Two conformal $\alpha$-regular domains $\Omega_1=\varphi_1(\mathbb{D})$ and $\Omega_2=\varphi_2(\mathbb{D})$ represent an conformal $\alpha$-regular pair if  and only if  
$$
E_{\alpha}(\varphi_1,\varphi_2)=\left(\iint\limits_{\mathbb D}\max\left\{\frac{|\varphi_1'(z)|^{\alpha}}{|\varphi_2'(z)|^{\alpha-2}},\frac{|\varphi_2'(z)|^{\alpha}}{|\varphi_1'(z)|^{\alpha-2}}\right\}~dxdy\right)^{\frac{1}{\alpha}}<\infty.
$$
\end{prop}

A property of conformal regular pairs is not very restrictive. For example any two Lipschitz simply connected bounded domains represent a conformal regular pair. Any two quasidiscs (images of discs under quasiconformal homeomorphisms of the plane) also represent a conformal regular pair.

The main result of this paper is

\begin{thm} 
\label{mainthm}
Let $\Omega_1=\varphi_1(\mathbb{D})$ , $\Omega_2=\varphi_2(\mathbb{D})$ be a conformal $\alpha$-regular pair.
Then for any $n\in\mathbb N$
\begin{equation}
\label{estimate 2}
|\lambda_{n}[\Omega_{1}]-\lambda_{n}[\Omega_{2}]|
\leq 2c_n \left[C(\alpha)\right]^2 E_{\alpha}(\varphi_1,\varphi_2)\||\varphi_1'|-|\varphi_2'|\mid L^{2}(\mathbb D)\|\,,
\end{equation}
where $\Omega_{1}=\varphi_1(\mathbb D)$, $\Omega_{2}=\varphi_2(\mathbb D)$,
and 
\begin{equation}
\label{c_n}
c_{n}=\max\{\lambda_{n}^{2}[\Omega_{1}],\lambda_{n}^{2}[\Omega_{2}]\}\,
\end{equation}
\end{thm}

\begin{rem}
The constant $\tilde{C}(\alpha)$ is the best constant for corresponding weighted Poincar\'e-Sobolev inequalities in the unit disc.
\end{rem}

In Section~5 we consider in more detail the case in which
$\Omega_1$ and $\Omega_2$ are quasidiscs.

The estimate for $|\lambda_{n}[\Omega_{1}]-\lambda_{n}[\Omega_{2}]|$ can also be given in terms of the measure variation:
\begin{multline}
|\lambda_{n}[\Omega_{1}]-\lambda_{n}[\Omega_{2}]|\leq 2c_n \left[C\left(\alpha\right)\right]^2 E_{\alpha}(\varphi_1,\varphi_2)\times\\
\times\Big(\left[{\rm meas}\,(\varphi_1(\mathbb D^+))-{\rm meas}\,(\varphi_2(\mathbb D^+))\right]+
\left[{\rm meas}\,(\varphi_2(\mathbb D^-))-{\rm meas}\,(\varphi_1(\mathbb D^-))\right]\Big)^{\frac12}\,.
\nonumber
\end{multline}
where
\begin{equation}\label{D+-}
\mathbb D^+=\{z\in\mathbb D:J_{\varphi_1}(z)\geq J_{\varphi_2}(z) \}\,, ~~~\mathbb D^-=\{z\in\mathbb D:J_{\varphi_1}(z)< J_{\varphi_2}(z) \}
\end{equation}
and $J_{\varphi_1}, J_{\varphi_2}$ are the Jacobians of the mappings $\varphi_1$, $\varphi_2$ respectively.

The inequality (\ref{estimate 2}) hold for any $\varphi_1,\varphi_2$ under consideration, but they are non-trivial only if 
$$
\||\varphi_1'|-|\varphi_2'|\mid L^{2}(\mathbb D)\|< \left(\sqrt{c_n} 2\left[C\left(\alpha\right)\right]^2 E_{\alpha}(\varphi_1,\varphi_2)\right)^{-1}
$$
respectively, because the inequality $|\lambda_{n}[\Omega_{1}]-\lambda_{n}[\Omega_{2}]|<\sqrt{c_n}$ obviously holds for any $\lambda_{n}[\Omega_{1}], \lambda_{n}[\Omega_{2}]$.

In this article we adopt an investigation method based on the theory of composition operators \cite{U1,VU1}. This method was applied in \cite{BGU15} to the spectral stability problem for the Dirichlet Laplacian and to the eigenvalue problem in \cite{GU15}.

Let $\Omega\subset\mathbb{C}$ be an arbitrary bounded simply connected planar
domain with a smooth boundary. Consider the classic eigenvalue problem
for the Neumann Laplacian in $\Omega$
\[
\begin{cases}
-\Delta_{w}g(w)=\lambda g(w),\,\, w\in\Omega,\\
\,\,\,\,\, \frac{\partial g}{\partial n} \big|_{\partial\Omega}=0\,,
\end{cases}
\]
where $$
\Delta_w=\frac{\partial^2}{\partial u^2}+\frac{\partial^2}{\partial v^2},\,\,\,w=u+iv\,.
$$
By the Riemann Mapping Theorem there exists a conformal mapping
$\varphi:\mathbb{D}\to\Omega$ from the unit disc $\mathbb{D}$ to
$\Omega$. Then, by the chain rule for the function $f(z)=g\circ\varphi(z)$,
we have
\begin{multline*}
\Delta_{z}f(z)=\Delta_{z}(g\circ\varphi(z))=(\Delta_{w}g)(\varphi(z))\cdot|\varphi'(z)|^{2}\\
=-\lambda g(\varphi(z))\cdot|\varphi'(z)|^{2}=-\lambda|\varphi'(z)|^{2}f(z).
\end{multline*}
 Here $\Omega\ni w=\varphi(z),\,\,\, z\in\mathbb{D}$. Hence we obtain
the weighted eigenvalue problem for the Neumann Laplacian in the unit
disc $\mathbb{D}$
\[
\begin{cases}
-\Delta f(z)=\lambda h(z)f(z),\,\, z\in\mathbb{D}\,,\\
\,\,\,\,\, \frac{\partial f}{\partial n}\big|_{\partial\mathbb D}=0\,,
\end{cases}
\]
 where
\begin{equation}\label{h}
h(z):=|\varphi'(z)|^{2}=J_{\varphi}(z)=\frac {\lambda_{\mathbb{D}}^{2}(z)}{\lambda_{\Omega}^{2}(\varphi(z))}
\end{equation}
 is the hyperbolic (conformal) weight defined by the conformal mapping
$\varphi:\mathbb{D}\to\Omega$. Here
$\lambda_{\mathbb{D}}$ and $\lambda_{\Omega}$ are hyperbolic metrics
in $\mathbb{D}$ and $\Omega$ respectively \cite{BM}.

That means that the eigenvalue problem in $\Omega$ is equivalent to
the weighted eigenvalue problem in the unit disc $\mathbb{D}$.

In the sequel we consider the weak formulation of the weighted eigenvalue problem for Neumann Laplacian, namely:
\begin{equation}
\iint\limits _{\mathbb{D}}(\nabla f(z)\cdot\nabla\overline{{g(z)}})~dxdy=\lambda\iint\limits _{\mathbb{D}}h(z)f(z)\overline{{g(z)}}~dxdy,\,\,~~\forall g\in W^{1,2}(\mathbb{D},h,1).\label{WEn}
\end{equation}

The method suggested to study the weighted eigenvalue problem
for the Dirichlet Laplacian is based on the theory of composition operators
\cite{U1,VU1} and the ``\,transfer\,'' diagram
suggested in \cite{GGu}. Universal hyperbolic weights for weighted
Sobolev inequalities were introduced in \cite{GU4} (see also \cite{LM1}).

\section{The weighted eigenvalue problem}

Let $\Omega\subset\mathbb{C}$ be an open set on the complex plane.
The Sobolev space $W^{1,p}(\Omega)$, $1\leq p\le\infty$, is
the normed space of all locally integrable weakly differentiable functions
$f:\Omega\to\mathbb{R}$ with finite norm given by
\begin{align*}
\|f\mid W^{1,p}(\Omega)\|=\biggr(\iint\limits _{\Omega}|f(z)|^{p}\, dxdy\biggr)^{1/p}+\biggr(\iint\limits _{\Omega}|\nabla f(z)|^{p}\, dxdy\biggr)^{1/p}, \,\,1\leq p<\infty,
\\
\|f\mid W^{1,\infty}(\Omega)\|={\begin{array}{c}\\{\rm ess~sup}\\{z\in\Omega}\end{array}} |f(z)|+{\begin{array}{c}\\{\rm ess~sup}\\{z\in\Omega}\end{array}} |\nabla f(z)|.
\end{align*}

The seminormed Sobolev space $L^{1,p}(\Omega)$, $1\leq p\le\infty$, is the space of all locally integrable weakly differentiable functions
$f:\Omega\to\mathbb{R}$ with finite seminorm given by
\begin{align*}
\|f\mid L^{1,p}(\Omega)\|=\biggr(\iint\limits _{\Omega}|\nabla f(z)|^{p}\, dxdy\biggr)^{1/p}, \,\,1\leq p<\infty,
\\
\|f\mid L^{1,\infty}(\Omega)\|={\begin{array}{c}\\{\rm ess~sup}\\{z\in\Omega}\end{array}} |\nabla f(z)|.
\end{align*}

The weighted Lebesgue space $L^{p}(\Omega,h)$, $1\leq p<\infty$,
is the space of all locally integrable functions with the finite
norm
$$
\|f\mid L^{p}(\Omega,h)\|=\bigg(\iint\limits _{\Omega}|f(z)|^{p}h(z)~dxdy\biggr)^{\frac{1}{p}}.
$$
Here the weight $h:\Omega\to\mathbb{R}$ is a non-negative measurable function.

We define the weighted Sobolev space $W^{1,p}(\Omega,h,1)$, $1\leq p<\infty$,
as the normed space of all locally integrable weakly differentiable functions
$f:\Omega\to\mathbb{R}$ with the finite norm given by
$$
\|f\mid W^{1,p}(\Omega,h,1)\|=\|f\mid L^p(\Omega,h)\|+\|\nabla f\mid L^p(\Omega)\|.
$$

We denote
\begin{multline}
f_{{\mathbb D}, h}=\frac{1}{m_h(\mathbb D)}\iint\limits_{\mathbb D}f(z)h(z)~dxdy=g_{\Omega}=\frac{1}{|\Omega|}\iint\limits_{\Omega} g(w)~dudv,\\
f(z)=g(\varphi(z)), \,\,\, w=\varphi(z).
\nonumber
\end{multline}
Here
$$
m_h(\mathbb D)=\iint\limits_{\mathbb D}h(z)~dxdy=\iint\limits_{\mathbb D}J_{\varphi}(z)~dxdy=|\Omega|.
$$

Recall that for conformal regular domains for any $2\le q<\infty$ the Sobolev type inequality
$$
\|g-g_{\Omega}\mid L^{q}(\Omega)\|\leq C(q)\|\nabla f\mid L^{2}(\Omega)\|
$$
holds for any function $g\in W^{1,2}(\Omega)$ (see \cite{GU15}).

Using this inequality we prove the following weighted Poincar\'e inequality for the unit disc:

\begin{thm} \label{thm:BoundEm} Let  $\Omega\subset\mathbb{C}$ be a conformal $\alpha$-regular
domain and $\varphi:\mathbb{D}\to\Omega$ be a conformal mapping.

Then for any function $u\in W^{1,2}(\mathbb{D},h,1)$ the
inequality
\[
\|f- f_{{\mathbb D}, h}\mid L^{2}(\mathbb{D},h)\|\leq K^{*}\|f\mid L^{1,2}(\mathbb{D})\|
\]
holds.

Here $h$ is the hyperbolic
(conformal) weight defined by equality $(\ref{h})$. The exact constant $K^{*}$ is equal to the exact
constant in the inequality
$$\|g-g_{\Omega}\mid L^{2}(\Omega)\|\leq K\|g\mid L^{1,2}(\Omega)\|\,,~~~~
\forall g\in W^{1,2}(\Omega)\,.$$
\end{thm}

\begin{proof} Since $\varphi^{-1}:\Omega\to\mathbb{D}$ is a conformal
mapping, the composition operator
\[
(\varphi^{-1})^{*}:L^{1,2}(\mathbb{D})\to L^{1,2}(\Omega),\,\,\,(\varphi^{-1})^{*}(f)=f\circ\varphi^{-1},
\]
is an isometry \cite{GU4}. 

Let $f\in C^{1}(\mathbb{D})$,
then $g=(\varphi^{-1})^{*}(f)=f\circ\varphi^{-1}\in C^{1}(\Omega)$. 
So, because $\Omega$ is a conformal regular domain, then for the function $g\in C^{1}(\Omega)$ the Poincar\'e inequality
\begin{equation}
\label{PoIn}
\|g-g_{\Omega}\mid L^{2}(\Omega)\|\leq K^{*}\|g\mid L^{1,2}(\Omega)\|
\end{equation}
holds with the exact constant $K^{*}=\sqrt{1/\lambda_2[\Omega]}$. Hence, using
the ``\,transfer\,'' diagram \cite{GGu} we
obtain
\begin{multline*}
\|f-f_{{\mathbb D},h}\mid L^{2}(\mathbb{D},h)\|=\biggl(\iint\limits _{\mathbb{D}}|f(z)-f_{{\mathbb D},h}|^{2}h(z)~dxdy\biggr)^{\frac{1}{2}}\\
=\biggl(\iint\limits _{\mathbb{D}}|f(z)-f_{{\mathbb D},h}|^{2}J(z,\varphi)(z)~dxdy\biggr)^{\frac{1}{2}}
=\biggl(\iint\limits _{\Omega}|g(w)-g_{\Omega}|^{2}~dudw\biggr)^{\frac{1}{2}}\\
\leq K^{*}\biggl(\iint\limits _{\Omega}|\nabla g(w))|^{2}~dudw\biggr)^{\frac{1}{2}}
=K^{*}\biggl(\iint\limits _{\mathbb{D}}|\nabla f(z))|^{2}~dxdy\biggr)^{\frac{1}{2}}
=K^{*}\|f\mid L^{1,2}(\mathbb{D})\|.
\end{multline*}
Approximating an arbitrary function $f\in W^{1,2}(\mathbb{D},h,1)$
by functions in the space $C^{1}(\mathbb D)$ we obtain that the inequality
\[
\|f-f_{{\mathbb D},h}\mid L^{2}(\mathbb{D},h)\|\leq K^{*}\|f\mid L^{1,2}(\mathbb{D})\|
\]
holds for any function $f\in W^{1,2}(\mathbb{D},h,1)$.

\end{proof}

By Theorem \ref{thm:BoundEm} it immediately follows that the spectral weighted eigenvalue problem (\ref{WEn}) with hyperbolic (conformal) weights $h$ in the unit disc $\mathbb D$
is equivalent to the eigenvalue problem in the domain $\Omega=\varphi(\mathbb D)$ (see also, for example \cite{LM1}) and
\begin{equation}\label{equality}
\lambda_{n}[h]=\lambda_{n}[\Omega], \,\,\, n\in\mathbb N\,.
\end{equation}

Hence the spectrum of  the weighted eigenvalue problem (\ref{WEn})  with hyperbolic (conformal) weights $h$
is discrete and can be written in the form of a non-decreasing sequence
\[
0=\lambda_{1}[h]<\lambda_{2}[h]\leq...\leq\lambda_{n}[h]\leq...\,,
\]
where each eigenvalue is repeated as many times as its multiplicity.

For weighted eigenvalues (eigenvalues in $\Omega$) we have also the following properties:

(i)
\[
\lim\limits _{n\to\infty}\lambda_{n}[h]=\lim\limits _{n\to\infty}\lambda_{n}[\Omega]=\infty\,,
\]

(ii) for each $n\in\mathbb{N}$
\begin{equation}
\lambda_n[\Omega]=\lambda_{n}[h]=\inf\limits _{\substack{L\subset W^{1,2}(\Omega)\\
\dim L=n
}
}\sup\limits _{\substack{f\in L\\
f\ne0
}
}\frac{\iint\limits _{\Omega}|\nabla f|^{2}~dxdy}{\iint\limits _{\Omega}|f|^{2}~dxdy}=
\inf\limits _{\substack{L\subset W^{1,2}(\mathbb D,h,1)\\
\dim L=n
}
}\sup\limits _{\substack{f\in L\\
f\ne0
}
}\frac{\iint\limits _{\mathbb D}|\nabla f|^{2}~dxdy}{\iint\limits _{\mathbb D}|f|^{2}h(z)~dxdy}\label{MinMax}
\end{equation}
 (Min-Max Principle), and
\begin{equation}
\lambda_{n}[h]=\sup\limits _{\substack{f\in M_{n}\\
f\ne0
}
}\frac{\iint\limits _{\mathbb D}|\nabla f|^{2}~dxdy}{\iint\limits _{\mathbb D}|f|^{2}h(z)~dxdy}\label{MaxPr}
\end{equation}
 where
\[
M_{n}={\rm span}\,\{\psi_{1}[h],\psi_{2}[h],...\psi_{n}[h]\}
\]
and $\{\psi_{k}[h]\}_{k=1}^{\infty}$ is an orthonormal (in the space $W^{1,2}(\mathbb D,h,1)$) set of eigenfunctions
corresponding to the eigenvalues $\{\lambda_{k}[h]\}_{k=1}^{\infty}$.

(iii) $\lambda_1[h]=0$ and $\psi_1=\frac{1}{\sqrt{m_h(\mathbb D)}}$. For $n\geq 2$ alongside with (\ref{MaxPr}) we have
\begin{equation}
\lambda_{n}[h]=\sup\limits _{\substack{f\in M_n\\f\ne0}}
\frac{\iint\limits _{\mathbb D}|\nabla f|^{2}~dxdy}{\iint\limits _{\mathbb D}|f-f_{\mathbb D,h}|^{2}h(z)~dxdy}.
\label{MaxPrL}
\end{equation}
(It may happen that the above fraction takes the form $\frac{0}{0}$. In this case we assume that $\frac{0}{0}=0$.)

Indeed, let 
\[
\tilde{M}_{n}={\rm span}\,\{\psi_{2}[h],...\psi_{n}[h]\}.
\]
Then, since $f\in M_n$ has the form $f=g-c$ where $g\in \tilde{M}_{n}$ and $c\in\mathbb R$, by (\ref{MaxPr})
\begin{multline}
\lambda_{n}[h]=\sup\limits _{\substack{g\in \tilde{M}_{n}\\g\ne0}}\sup\limits_{c\in\mathbb R}
\frac{\iint\limits _{\mathbb D}|\nabla (g-c)|^{2}~dxdy}{\iint\limits _{\mathbb D}|g-c|^{2}h(z)~dxdy}
=
\sup\limits _{\substack{g\in \tilde{M}_{n}\\g\ne0}}
\frac{\iint\limits _{\mathbb D}|\nabla g|^{2}~dxdy}{\inf\limits_{c\in\mathbb R}\iint\limits _{\mathbb D}|g-c|^{2}h(z)~dxdy}\\
=
\sup\limits _{\substack{g\in \tilde{M}_{n}\\g\ne0}}
\frac{\iint\limits _{\mathbb D}|\nabla g|^{2}~dxdy}{\iint\limits _{\mathbb D}|g-g_{\mathbb D,h}|^{2}h(z)~dxdy}
=
\sup\limits _{\substack{g\in \tilde{M}_{n}\\g\ne0}}\sup\limits_{c\in\mathbb R}
\frac{\iint\limits _{\mathbb D}|\nabla (g-c)|^{2}~dxdy}{\iint\limits _{\mathbb D}|(g-c)-(g-c)_{\mathbb D,h}|^{2}h(z)~dxdy}\\
=
\sup\limits _{\substack{f\in M_n\\f\ne0}}
\frac{\iint\limits _{\mathbb D}|\nabla f|^{2}~dxdy}{\iint\limits _{\mathbb D}|f-f_{\mathbb D,h}|^{2}h(z)~dxdy},
\nonumber
\end{multline}
because $g-g_{\mathbb D,h}=(g-c)-(g-c)_{\mathbb D,h}$ for any $c\in\mathbb R$.

\section{The $L^{1,2}$-seminorm estimates}

We consider two weighted eigenvalue problems in the unit disc $\mathbb D\subset\mathbb C$:
\[
\iint\limits _{\mathbb D}(\nabla f(z)\cdot\nabla\overline{{g(z)}})~dxdy=\lambda\iint\limits _{\mathbb D}h_{1}(z)f(z)\overline{g(z)}~dxdy\,,\,\,~~\forall g\in W^{1,2}(\mathbb D,h_1,1).
\]
 and
\[
\iint\limits _{\mathbb D}(\nabla f(z)\cdot\nabla\overline{{g(z)}})~dxdy=\lambda\iint\limits _{\mathbb D}h_{2}(z)f(z)\overline{{g(z)}}~dxdy\,,\,\,~~\forall g\in W^{1,2}(\mathbb D,h_2,1).
\]

The aim of this section is to estimate the ``\,distance\,''
between weighted eigenvalues $\lambda_{n}[h_{1}]$ and $\lambda_{n}[h_{2}]$.
\begin{lem}
\label{lem:TwoWeight} Let $\mathbb D\subset\mathbb{C}$ be the unit disc
and let $h_{1}$, $h_{2}$ be conformal weights on $\mathbb D$. Suppose that there exists a constant $B>0$ such that
\begin{multline}
\max\left\{\iint\limits _{\mathbb D}|h_{1}(z)-h_{2}(z)||f-f_{\mathbb D,h_1}|^{2}~dxdy,
\iint\limits _{\mathbb D}|h_{1}(z)-h_{2}(z)||f-f_{\mathbb D,h_2}|^{2}~dxdy\right\}\\
\leq B\iint\limits _{\mathbb D}|\nabla f|^{2}~dxdy,\,\,
\forall f\in L^{1,2}(\mathbb D).
\label{EqvWW}
\end{multline}

Then for any $n\in\mathbb{N}$
\begin{equation}
|\lambda_{n}[h_{1}]-\lambda_{n}[h_{2}]|\leq \frac{B\tilde c_n}{1+B\sqrt{\tilde c_n}}< B\tilde c_n\,,\label{LemEq}
\end{equation}
where
\begin{equation}
\label{tilde c_n}
\tilde c_n=\max\{\lambda_{n}^{2}[h_{1}],\lambda_{n}^{2}[h_{2}]\} \,.
\end{equation}
\end{lem}

\begin{proof} Let
\[
M_{n}^{(1)}={\rm span}\,\{\psi_{1}[h_{1}],...\psi_{n}[h_{1}]\}.
\]
Since
\begin{equation}
\iint\limits_{\mathbb D}h_1(z)|f-f_{\mathbb D, h_1}|^2~dxdy=
\inf\limits_{c\in\mathbb R}\iint\limits_{\mathbb D}h_1(z)|f-c|^2~dxdy
\leq
\iint\limits_{\mathbb D}h_1(z)|f-f_{\mathbb D, h_2}|^2~dxdy,
\nonumber
\end{equation}
by (\ref{MaxPrL}) we get
\[
\lambda_{n}[h_{1}]=\sup\limits _{\substack{f\in M_{n}^{(1)}\\
f\ne0
}
}\frac{\iint\limits _{\mathbb D}|\nabla f|^{2}~dxdy}{\iint\limits _{\mathbb D}h_{1}(x)|f-f_{\mathbb D,h_1}|^{2}~dxdy}\geq
\sup\limits _{\substack{f\in M_{n}^{(1)}\\
f\ne0
}
}
\frac{\iint\limits _{\mathbb D}|\nabla f|^{2}~dxdy}{\iint\limits _{\mathbb D}h_{1}(x)|f-f_{\mathbb D,h_2}|^{2}~dxdy}.
\]
Hence, by (\ref{EqvWW}),
\begin{multline*}
\lambda_{n}[h_{1}]  \geq\sup\limits _{\substack{f\in M_{n}^{(1)}\\
f\ne0
}
}\frac{\iint\limits _{\mathbb D}|\nabla f|^{2}~dxdy}{\iint\limits _{\mathbb D}h_{2}(z)|f-f_{\mathbb D,h_2}|^{2}~dxdy+\iint\limits _{\mathbb D}|h_{1}(z)-h_{2}(z)||f-f_{\mathbb D,h_2}|^{2}~dxdy}
\\
\geq\sup\limits _{\substack{f\in M_{n}^{(1)}\\
f\ne0
}
}\frac{\iint\limits _{\mathbb D}|\nabla f|^{2}~dxdy}{\iint\limits _{\mathbb D}h_{2}(z)|f-f_{\mathbb D,h_2}|^{2}~dxdy+B\iint\limits _{\mathbb D}|\nabla f|^{2}~dxdy}.
\end{multline*}

Since
$$
\iint\limits _{\mathbb D}h_{2}(z)|f-f_{\mathbb D,h_2}|^{2}~dxdy=
\inf\limits_{c\in\mathbb R}\iint\limits _{\mathbb D}h_{2}(z)|f-c|^{2}~dxdy
\leq\iint\limits _{\mathbb D}h_{2}(z)|f|^{2}~dxdy,
$$
we get
\[
\begin{split}
\lambda_n[h_1] & \geq
\sup\limits _{\substack{f\in M_{n}^{(1)}\\
f\ne0
}
}\frac{\iint\limits _{\mathbb D}|\nabla f|^{2}~dxdy}{\iint\limits _{\mathbb D}h_{2}(z)|f|^{2}~dxdy+B\iint\limits _{\mathbb D}|\nabla f|^{2}~dxdy}\\
& = \sup\limits _{\substack{f\in M_{n}^{(1)}\\
f\ne0
}
}\frac{\iint\limits _{\mathbb D}|\nabla f|^{2}~dxdy}{\iint\limits _{\mathbb D}h_{2}(z)|f|^{2}~dxdy}\cdot\frac{1}{1+B\frac{\iint\limits _{\mathbb D}|\nabla f|^{2}~dxdy}{\iint\limits _{\mathbb D}h_{2}(z)|f|^{2}~dxdy}}\\
& \geq\sup\limits _{\substack{f\in M_{n}^{(1)}\\
f\ne0
}
}\frac{\iint\limits _{\mathbb D}|\nabla f|^{2}~dxdy}{\iint\limits _{\mathbb D}h_{2}(z)|f|^{2}~dxdy}\cdot\inf\limits _{\substack{f\in M_{n}^{(1)}\\
f\ne0
}
}\frac{1}{1+B\frac{\iint\limits _{\mathbb D}|\nabla f|^{2}~dxdy}{\iint\limits _{\mathbb D}h_{2}(z)|f|^{2}~dxdy}}\\
& =\sup\limits _{\substack{f\in M_{n}^{(1)}\\
f\ne0
}
}\frac{\iint\limits _{\mathbb D}|\nabla f|^{2}~dxdy}{\iint\limits _{\mathbb D}h_{2}(z)|f|^{2}~dxdy}\cdot\frac{1}{1+B\sup\limits _{\substack{f\in M_{n}^{(1)}\\
f\ne0}}\frac{\iint\limits _{\mathbb D}|\nabla f|^{2}~dxdy}{\iint\limits _{\mathbb D}h_{2}(z)|f|^{2}~dxdy}}.
\end{split}
\]
 Since the function $F(t)={t}/{(1+Bt)}$ is non-decreasing on $[0,\infty)$
and by (\ref{MinMax})
\[
\sup\limits _{\substack{f\in M_{n}^{(1)}\\
f\ne0
}
}\frac{\iint\limits _{\mathbb D}|\nabla f|^{2}~dxdy}{\iint\limits _{\mathbb D}h_{2}(z)|f|^{2}~dxdy}\geq\lambda_{n}[h_{2}],
\]
 it follows that
\[
\lambda_{n}[h_{1}]\geq\frac{\lambda_{n}[h_{2}]}{1+B\lambda_{n}[h_{2}]}=\lambda_{n}[h_{2}]- \frac{B\lambda_{n}^2[h_{2}]}{1+B\lambda_{n}[h_{2}]}.
\]
 Hence
\begin{equation}
\lambda_{n}[h_{1}]-\lambda_{n}[h_{2}]\geq-\frac{B\lambda_{n}^2[h_{2}]}{1+B\lambda_{n}[h_{2}]}\geq-\frac{B\tilde c_n}{1+B\sqrt{\tilde c_n}}.\label{LemEqR}
\end{equation}
 For similar reasons
\[
\lambda_{n}[h_{2}]-\lambda_{n}[h_{1}]\geq -\frac{B\lambda^2_{n}[h_{1}]}{1+B\lambda_{n}[h_{1}]}
\]
 or
\begin{equation}
\lambda_{n}[h_{1}]-\lambda_{n}[h_{2}]\leq \frac{B\lambda^2_{n}[h_{1}]}{1+B\lambda_{n}[h_{1}]}\le \frac{B\tilde c_n}{1+B\sqrt{\tilde c_n}}.\label{LemEqL}
\end{equation}

Inequalities (\ref{LemEqR}) and (\ref{LemEqL}) imply inequality
(\ref{LemEq}). \end{proof}

Now we estimate the constant $B$ in Lemma \ref{lem:TwoWeight} in terms of ``\,vicinity\,'' between weights.

By Theorem \ref{thm:BoundEm} 

\begin{equation}\label{Sobolev}
\|f-f_{\mathbb D,h_k}\mid L^{q}(\mathbb{D},h_k)\|\leq C_k(q)\|\nabla f\mid L^{2}(\mathbb{D})\|
\end{equation}
for any function $f\in W^{1,2}(\mathbb{D},h_k,1)$, where $h_k$, $k=1,2$, are the conformal weights defined be equality (\ref{h}). Here $C_k(q)$
are the best possible constants in these inequalities, that depend on $\alpha$ only.

\begin{lem}
\label{lem:TwoWeiPol} Let   $h_{1}$, $h_{2}$ be conformal weights on
$\mathbb{D}$ such that
\begin{multline}
d_{s}(h_{1},h_{2}):= \|(h_1-h_2)(\min\{h_1,h_2\})^{\frac{1-s}{s}}\mid L^{s}(\mathbb{D})\|\\
=\left(\iint\limits_{\mathbb D}(h_1-h_2)^s(\min\{h_1,h_2\})^{1-s}~dxdy\right)^{\frac{1}{s}}<\infty  
\label{EqvWWpol}
\end{multline}
 for some $1<s\le\infty$.

Then  inequality $(\ref{EqvWW})$ holds with the constant
\begin{equation}
B=\Big[C\Big(\frac{2s}{s-1}\Big)\Big]^2\,d_{s}(h_{1},h_{2})\,\,\,C\Big(\frac{2s}{s-1}\Big)=
\max\left\{C_1\Big(\frac{2s}{s-1}\Big),C_2\Big(\frac{2s}{s-1}\Big)\right\}\,.
\label{Lem2Es}
\end{equation}
\end{lem}

\begin{proof}
By the H\"older inequality and Sobolev inequality (\ref{Sobolev}) we get for $k=1,2$
\begin{multline}
\iint\limits _{\mathbb{D}}|h_{1}(z)-h_{2}(z)||f(z)-f_{\mathbb D,h_k}|^{2}~dxdy
\\
=\iint\limits _{\mathbb{D}}|h_{1}(z)-h_{2}(z)|h_k^{\frac{1-s}{s}}(z)h_k^{\frac{s-1}{s}}(z)|f(z)-f_{\mathbb D,h_k}|^{2}~dxdy
\\
\leq\|\left(h_{1}-h_{2}\right)h_k^{\frac{1-s}{s}}\mid L^s(\mathbb D)\|\left(\iint\limits _{\mathbb{D}}h_k(z)|f(z)-f_{\mathbb D,h_k}|^{\frac{2s}{s-1}}dxdy\right)^{\frac{s-1}s}
\\
\leq \Big[C_k\Big(\frac{2s}{s-1}\Big)\Big]^2\,d_{s}(h_{1},h_{2}) \iint\limits _{\mathbb{D}}|\nabla f(z)|^{2}dxdy\,.
\nonumber
\end{multline}
\end{proof}

By the two previous lemmas we get immediately the main result for the difference of weighted eigenvalues:
\begin{thm}
\label{thm:TwoWW} Let  $h_{1}$, $h_{2}$ be conformal weights on
$\mathbb{D}$. Suppose that $d_{s}(h_{1},h_{2})<\infty$ for some $s>1$.

Then, for every $n\in\mathbb{N}$,
\[
|\lambda_{n}[h_{1}]-\lambda_{n}[h_{2}]|\leq \tilde c_{n}\Big[C\Big(\frac{2s}{s-1}\Big)\Big]^2 d_{s}(h_{1},h_{2})\,.
\]
\end{thm}

\section{On ``distances" $d_{s}(h_{1},h_{2})$ for hyperbolic (conformal) weights $h_{1},h_{2}$}

Let us analyze ``\,distances\,'' $d_{s}(h_{1},h_{2})$ for hyperbolic (conformal) weights.

Recall that hyperbolic (conformal) weights $h_{1}(z),h_{2}(z)$ for bounded simply connected planar domains are Jacobians 
$J_{\varphi_{1}}(z)$, $J_{\varphi_{2}}(z)$ of conformal homeomorphisms
$$\varphi_{1}:\mathbb{D}\to\Omega_{1},\,\,\, \varphi_{1}:\mathbb{D}\to\Omega_{2}.
$$

Since $\Omega_{1},\Omega_{2}$ are
bounded domains the Jacobians $J_{\varphi_{1}}(z)$, $J_{\varphi_{2}}(z)$ are integrable,
i.~e. $\varphi_{1}',\varphi_{2}'\in L^{2}(\mathbb D)$. An example of the
unit disc without the interval $(0,1)$ on the horizontal axis demonstrates
that for general simply connected domains $\Omega$ the Jacobians of conformal homeomorphisms
$\varphi:\mathbb{D}\to\Omega$ need not be integrable in a power greater
than $1$. Hence the integrability of  Jacobians to the power $s>1$
is possible only under additional assumptions on $\Omega$.

In \cite{GU5} it was proved that such integrability  is possible only for domains with finite
geodesic diameter. Hence, for $s>1$, the quantity $d_{s}(h_{1},h_{2})$ is not necessary defined
for all pairs of conformal weights $h_{1},h_{2}$.

\begin{lem}\label{Lemma 1}
\label{thm:MesureEstimate}
Let $\varphi_{1}:\mathbb{D}\to\Omega_{1}$, $\varphi_{2}:\mathbb{D}\to\Omega_{2}$ be conformal homeomorphisms and
$h_{1},h_{2}$ be the corresponding conformal weights. Suppose that for some $2<p<\infty$
$$
E_{p}(\varphi_1,\varphi_2)=\left(\iint\limits_{\mathbb D}\max\left\{\frac{|\varphi_1'(z)|^p}{|\varphi_2'(z)|^{p-2}},\frac{|\varphi_2'(z)|^p}{|\varphi_1'(z)|^{p-2}}\right\}~dxdy\right)^{\frac{1}{p}}<\infty.
$$

Then for $s=\frac{2p}{p+2}$
\begin{equation} \label{measure}
d_{s}(h_{1},h_{2})\leq 2E_{p}(\varphi_1,\varphi_2)
\cdot\||\varphi_1'|-|\varphi_2'|\mid L^{2}(\mathbb D)\|<\infty.
\end{equation}
\end{lem}

\begin{proof}
By the definitions of $h_1,h_2$ and $d_{s}(h_{1},h_{2})$
\begin{multline}
\left[d_{s}(h_{1},h_{2})\right]^{s}=\iint\limits _{\mathbb{D}}\left|h_{1}(z)-h_{2}(z)\right|^{s}\left(\min\{h_1(z),h_2(z)\}\right)^{1-s}dxdy\\
=\iint\limits _{\mathbb{D}}\left||\varphi'_{1}(z)|^{2}-|\varphi'_{2}(z)|^{2}\right|^{s}\left(\min\{|\varphi'_{1}(z)|,|\varphi'_{2}(z)|\}\right)^{2(1-s)}dxdy\\
=
\iint\limits_{\mathbb{D}}\left||\varphi'_{1}(z)|+|\varphi'_{2}(z)|\right|^{s}\left(\min\{|\varphi'_{1}(z)|,|\varphi'_{2}(z)|\}\right)^{2(1-s)}\left||\varphi_{1}'(z)|-|\varphi_{2}'(z)|\right|^{s}dxdy.
\nonumber
\end{multline}
Applying to  the last integral the H\"older inequality  with $r=\frac{2}{s}$
($1\le r<2$ because $1<s\le 2$) and $r'=\frac{r}{r-1}=\frac{2}{2-s}$
we obtain
\begin{multline}
\left[d_{s}(h_{1},h_{2})\right]^{s}
\leq\left(\iint\limits _{\mathbb{D}}\left||\varphi_{1}'(z)|+|\varphi_{2}'(z)|\right|^{\frac{2s}{2-s}}\left(\min\{|\varphi'_{1}(z)|,|\varphi'_{2}(z)|\}\right)^{\frac{4(1-s)}{2-s}}dxdy\right)^{\frac{2-s}{2}}
\\
\times\left(\iint\limits _{\mathbb{D}}\left(|\varphi_{1}'(z)|-|\varphi_{2}'(z)|\right)^{2}dxdy\right)^{\frac{s}{2}}\\
\leq
2^s\left(\iint\limits_{\mathbb{D}}\frac{\max\left\{|\varphi_{1}'(z)|,|\varphi_{2}'(z)|\right\}^{\frac{2s}{2-s}}}{\min\left\{|\varphi_{1}'(z)|,|\varphi_{2}'(z)|\right\}^{\frac{4(s-1)}{2-s}}}dxdy\right)^{\frac{2-s}{2}}
\||\varphi_{1}'|-|\varphi_{2}'|\mid L^2(\mathbb D)\|^s.
\nonumber
\end{multline}
Since $s=\frac{2p}{p+2}$ we have
\begin{multline}
d_{s}(h_{1},h_{2})\leq 2\left(\iint\limits_{\mathbb{D}}\frac{\max\left\{|\varphi_{1}'(z)|,|\varphi_{2}'(z)|\right\}^{p}}{\min\left\{|\varphi_{1}'(z)|,|\varphi_{2}'(z)|\right\}^{p-2}}dxdy\right)^{\frac{1}{p}}
\||\varphi_{1}'|-|\varphi_{2}'|\mid L^2(\mathbb D)\|\\
=2\left(\iint\limits_{\mathbb D}\max\left\{\frac{|\varphi_1'(z)|^p}{|\varphi_2'(z)|^{p-2}},\frac{|\varphi_2'(z)|^p}{|\varphi_1'(z)|^{p-2}}\right\}~dxdy\right)^{\frac{1}{p}}
\||\varphi_{1}'|-|\varphi_{2}'|\mid L^2(\mathbb D)\|.
\end{multline}
\end{proof}

\begin{prop} Two conformal $\alpha$-regular domains $\Omega_1=\varphi_1(\mathbb{D})$ and $\Omega_2=\varphi_2(\mathbb{D})$ represent a conformal $\alpha$-regular pair if  and only if  
$$
E_{\alpha}(\varphi_1,\varphi_2)=\left(\iint\limits_{\mathbb D}\max\left\{\frac{|\varphi_1'(z)|^{\alpha}}{|\varphi_2'(z)|^{\alpha-2}},\frac{|\varphi_2'(z)|^{\alpha}}{|\varphi_1'(z)|^{\alpha-2}}\right\}~dxdy\right)^{\frac{1}{\alpha}}<\infty.
$$
\end{prop}

\begin{proof}
The condition $E_{\alpha}(\varphi_1,\varphi_2)<\infty$ means that domains $\Omega_1=\varphi_1(\mathbb D),\Omega_2=\varphi_2(\mathbb D)$ represent conformal $\alpha$-regular pair.
Indeed, a conformal mapping $\psi:\Omega_1\to\Omega_2$ can be written as a composition 
$$
\psi(w)=\varphi_2(\varphi_1^{-1}(w)).
$$
Then
$$
\psi'(w)=\varphi_2'(\varphi_1^{-1}(w))\cdot (\varphi_1^{-1})'(w).
$$
Using the change of variable formula we obtain
\begin{multline}
\iint\limits_{\Omega_1}|\psi'(w)|^{\alpha}~dudv=\iint\limits_{\mathbb D}|\psi'|^{\alpha}(\varphi_1(z))J_{\varphi_1}(z)~dxdy\\ 
=\iint\limits_{\mathbb D}|\varphi_2'(z_)|^{\alpha}|(\varphi^{-1}_1)'|^{\alpha}(\varphi_1(z))|\varphi_1'(z)|^2~dxdy
=\iint\limits_{\mathbb D}|\varphi_2'(z_)|^{\alpha}|\varphi_1'(z)|^{-{\alpha}}|\varphi_1'(z)|^2~dxdy\\
=\iint\limits_{\mathbb D}|\varphi_2'(z_)|^{\alpha}|\varphi_1'(z)|^{2-{\alpha}}~dxdy.
\nonumber
\end{multline}
The similar calculation is correct for the inverse mapping $\psi^{-1}:\Omega_2\to\Omega_1$. 
\end{proof}

Note that integral estimate (\ref{measure}) can be rewritten in terms of the measure variation.

\begin{lem}\label{Lemma 2}
Let $\varphi_{1}:\mathbb{D}\to\Omega_{1}$, $\varphi_{2}:\mathbb{D}\to\Omega_{2}$ be conformal homeomorphisms. Then
\begin{multline}
\||\varphi_1'|-|\varphi_2'|\mid L^{2}(\mathbb D)\|\\
\le
\Big(\left[{\rm meas}\,(\varphi_1(\mathbb D^+))-{\rm meas}\,(\varphi_2(\mathbb D^+))\right]+
\left[{\rm meas}\,(\varphi_2(\mathbb D^-))-{\rm meas}\,(\varphi_1(\mathbb D^-))\right]\Big)^{\frac12}\,,
\nonumber
\end{multline}
where the sets $\mathbb D^+$ and $\mathbb D^-$ are defined by equalities $(\ref{D+-})$.
\end{lem}

\begin{proof} By  using the elementary inequality $(a-b)^{2}\leq |a^{2}-b^{2}|$ for any
$a,b\ge0$ and the equality $|\varphi_{1}'(z)|^2=J_{\varphi}$ for conformal homeomorphisms we get
\begin{multline}
\iint\limits _{\mathbb{D}}\left(|\varphi_{1}'(z)|-|\varphi_{2}'(z)|\right)^{2}dxdy\\
\leq\iint\limits _{\mathbb{D}}\left|\left|\varphi_{1}'(z)\right|^{2}-\left|\varphi_{2}'(z)\right|^{2}\right|dxdy
=
\iint\limits _{\mathbb{D}}\left|J_{\varphi_{1}}(z)-J_{\varphi_{2}}(z)\right|dxdy\\
=\iint\limits _{\mathbb{D^+}}\left(J_{\varphi_{1}}(z)-J_{\varphi_{2}}(z)\right)dxdy
+\iint\limits _{\mathbb{D^-}}\left(J_{\varphi_{2}}(z)-J_{\varphi_{1}}(z)\right)dxdy\\
=\Big(\left[{\rm meas}\,(\varphi_1(\mathbb D^+))-{\rm meas}\,(\varphi_2(\mathbb D^+))\right]+
\left[{\rm meas}\,(\varphi_2(\mathbb D^-))-{\rm meas}\,(\varphi_1(\mathbb D^-))\right]\Big)\,.
\nonumber
\end{multline}
\end{proof}

By combining Lemma \ref{Lemma 1}, Theorem \ref{thm:TwoWW}, equality (\ref{equality}), by applying the triangle inequality and taking into account that
$\frac{2s}{s-1}=\frac{4p}{p-2}$ for $s=\frac{2p}{p+2}$, we obtain the main result of this paper:

{\bf Theorem \ref{mainthm}}. 
{\it Let $\Omega_1=\varphi_1(\mathbb{D})$ , $\Omega_2=\varphi_2(\mathbb{D})$ be a conformal $\alpha$-regular pair.
Then for any $n\in\mathbb N$
\begin{equation}
|\lambda_{n}[\Omega_{1}]-\lambda_{n}[\Omega_{2}]|
\leq 2c_n \left[C\left(\frac{4\alpha}{\alpha-2}\right)\right]^2 E_{\alpha}(\varphi_1,\varphi_2)\||\varphi_1'|-|\varphi_2'|\mid L^{2}(\mathbb D)\|\,,
\end{equation}
where 
\begin{equation}
c_{n}=\max\{\lambda_{n}^{2}[\Omega_{1}],\lambda_{n}^{2}[\Omega_{2}]\}\,.
\end{equation}
}

By Lemmas \ref{Lemma 1} and \ref{Lemma 2} follows the estimate in terms of the measure variation:
\begin{multline}
\label{main estimate 2}
|\lambda_{n}[\Omega_{1}]-\lambda_{n}[\Omega_{2}]|\leq 2c_n \left[C\left(\frac{4\alpha}{\alpha-2}\right)\right]^2 E_{\alpha}(\varphi_1,\varphi_2)\times\\
\times\Big(\left[{\rm meas}\,(\varphi_1(\mathbb D^+))-{\rm meas}\,(\varphi_2(\mathbb D^+))\right]+
\left[{\rm meas}\,(\varphi_2(\mathbb D^-))-{\rm meas}\,(\varphi_1(\mathbb D^-))\right]\Big)^{\frac12}\,.
\nonumber
\end{multline}

\section{Quasidiscs}

Now we describe a rather wide class of planar domains for which there exist conformal mappings with Jacobians of the class
$L^p(\mathbb D)$ for some $p>1$, i.e. with complex derivatives of the class
$L^p(\mathbb D)$ for some $p>2$.

\begin{defn}
A homeomorphism $\varphi:\Omega_1\to\Omega_2$
between planar domains is called $K$-quasiconformal if it preserves
orientation, belongs to the Sobolev class $W_{\operatorname{loc}}^{1,2}(\Omega_1)$
and its directional derivatives $\partial_{\alpha}$ satisfy the distortion inequality
$$
\max\limits_{\alpha}|\partial_{\alpha}\varphi|\leq K\min\limits_{\alpha}|\partial_{\alpha}\varphi|\,\,\,
\text{a.e. in}\,\,\, \Omega_1\,.
$$
\end{defn}
Infinitesimally, quasiconformal homeomorphisms transform circles to ellipses
with eccentricity uniformly bounded by $K$. If $K=1$ we recover
conformal homeomorphisms, while for $K>1$ planar quasiconformal mappings need
not be smooth.
\begin{defn}
A domain $\Omega$ is called a $K$-quasidisc if it is the image of the
unit disc $\mathbb{D}$ under a $K$-quasiconformal homeomorphism of
the complex plane $\mathbb C$ onto itself.
\end{defn}

It is well known that the boundary of any $K$-quasidisc $\Omega$
admits a $K^{2}$-quasi\-con\-for\-mal reflection \cite{GH1} and thus, for example,
any conformal homeomorphism $\varphi:\mathbb{D}\to\Omega$ can be
extended to a $K^{2}$-quasiconformal homeomorphism of the whole plane
to itself.

The boundaries of quasidiscs are called quasicircles. It is known that there are quasicircles for which no segment has finite length.
The Hausdorff dimension of quasicircles was first investigated by F. W. Gehring and J. V\"ais\"al\"a  \cite{GV73},
who proved that it can take all values in the interval $[1,2)$. S. Smirnov proved recently \cite{Smi10} that the Hausdorff dimension of
any $K$-quasicircle is at most $1+k^2$, where $k = (K-1)/(K +1)$.

Ahlfors's 3-point condition \cite{Ahl63} gives
a complete geometric characterization of quasicircles: a Jordan curve $\gamma$ in the plane is a quasicircle
if and only if for each two points $a, b$ in $\gamma$ the (smaller) arc between them has 
diameter comparable with $|a-b|$. This condition is easily checked for the snowflake.
On the other hand, every quasicircle can be obtained by an explicit snowflake-type
construction (see \cite{Roh01}).

For any planar $K$-quasiconformal homeomorphism $\varphi:\Omega_1\rightarrow\Omega_2$
the following sharp result is known: $J(z,\varphi)\in L_{loc}^{p}(\Omega_{1})$
for any $p<\frac{K}{K-1}$ (\cite{G1, A}).
\begin{prop}
\label{prop:confQuasidisc}Any conformal mapping $\varphi:\Omega_1\to\Omega_2$
of  $K$-quasidiscs $\Omega_1,\Omega_2$ belongs to $L^{1,p}(\Omega_1)$
for any $1\le p<\frac{2K^{2}}{K^2-1}$.
\end{prop}

\begin{proof} Any conformal mapping $\varphi:\Omega_1\to\Omega_2$ can be extended to a $K^2$ quasiconformal homeomorphism  $\psi$
of the whole plane to the whole plane by reflection.
Since the domain $\Omega_2$ is bounded, $\psi$ belongs to the class $L^p(\Omega_1)$ for any $1\le p<\frac{2K^2}{K^2-1}$ (\cite{G1}, \cite{A}).
Therefore $\varphi$ belongs to the same class.
\end{proof}

Denote, for $K\geq 1$, by $A_K$ the class of all $K$-quasidiscs.
Theorem \ref{mainthm} and Proposition \ref{prop:confQuasidisc}  imply the following statement.

\begin{thm}
For any $K\geq 1$ there exists $p>2$ and $M>0$ such that, for any quasidiscs $\Omega_1,\Omega_2\in A_K$ and conformal mappings 
$\varphi_k:\mathbb D\to\Omega_k$, k=1,2,
$|\varphi_1'|, |\varphi_2'|\in L^{p}(\mathbb D)$ and for any $n\in\mathbb{N}$
\begin{equation*}
|\lambda_{n}[\Omega_{1}]-\lambda_{n}[\Omega_{2}]|
\leq 2c_n M E_p(\varphi_1,\varphi_2)\||\varphi_1'|-|\varphi_2'|\mid L^{2}(\mathbb D)\|\,,
\end{equation*}
 where $c_{n}$ is defined by equality $(\ref{c_n})$.
\end{thm}

\begin{proof}
Since $\frac{2K^2}{K^2-1}>2$, by Proposition \ref{prop:confQuasidisc} there exists $2<p<\frac{2K^2}{K^2-1}$,
say $p=\frac{2K^2-1}{K^2-1}$, such that $|\varphi'_1|,|\varphi'_2|\in L^p(\mathbb D)$ and $E_p({\varphi_1,\varphi_2})<\infty$. Therefore, by 
Theorem \ref{mainthm} the statement follows with, say $p=\frac{2K^2-1}{K^2-1}$ and 
$$
M=\left[C\left(\frac{4p}{p-2}\right)\right]^2=\left[C\left(4(2K^2-1)\right)\right]^2.
$$ 
\end{proof}

\end{document}